\documentclass[12pt]{amsart}
\usepackage{a4wide}
\usepackage{graphicx,eepic, bm, times}
\usepackage{amsthm,amsmath,amsfonts,amssymb}
\usepackage{multirow}

\newtheorem{lemma}{Lemma}[section]

\newtheorem{thm}[lemma]{Theorem}
\newtheorem{cor}[lemma]{Corollary}
\theoremstyle{definition}

\newtheorem{example}[lemma]{Example}

\theoremstyle{remark}

\numberwithin{equation}{section} \numberwithin{table}{section}

\begin{document}

\title[Badly approximable numbers for sequences of balls]{Badly approximable numbers for sequences of balls}
\author{Simon Baker}
\address{Utrecht University, P.O Box 80125, 3508 TC Utrecht, The Netherlands. E-mail:
simonbaker412@gmail.com}

\date{\today}

\begin{abstract}
It is a classical result from Diophantine approximation that the set of badly approximable numbers has Lebesgue measure zero. In this paper we generalise this result to more general sequences of balls. 

Given a countable set of closed $d$-dimensional Euclidean balls $\{B(x_{i},r_{i})\}_{i=1}^{\infty},$ we say that $\alpha\in \mathbb{R}^{d}$ is a badly approximable number with respect to $\{B(x_{i},r_{i})\}_{i=1}^{\infty}$ if there exists $\kappa(\alpha)>0$ and $N(\alpha)\in\mathbb{N}$ such that $\alpha\notin B(x_{i},\kappa(\alpha)r_{i})$ for all $i\geq N(\alpha)$. Under natural conditions on the set of balls, we prove that the set of badly approximable numbers with respect to $\{B(x_{i},r_{i})\}_{i=1}^{\infty}$ has Lebesgue measure zero. Moreover, our approach yields a new proof that the set of badly approximable numbers has Lebesgue measure zero. 
\end{abstract}

\maketitle

\section{Introduction}
Diophantine approximation is primarily concerned with how well real numbers can be approximated by rationals. A classical theorem due to Dirichlet states that for any $\alpha=(\alpha_{1},\ldots,\alpha_{d})\in \mathbb{R}^{d}$ and $Q\in \mathbb{N},$ there exists $(p_{1},\ldots,p_{d})\in \mathbb{Z}^{d}$ and $1\leq q<Q^{d}$ such that
$$\max_{1\leq i\leq d}|q\alpha_{i}-p_{i}|\leq \frac{1}{Q} .$$The proof of Dirichlet's theorem is a straightforward application of the pigeon hole principle. It is an immediate consequence of Dirichlet's theorem that if $\alpha_{i}$ is irrational for some $1\leq i\leq d,$ then there are infinitely many distinct $d$-tuples $(\frac{p_{1}}{q},\ldots, \frac{p_{d}}{q})\in\mathbb{Q}^{d}$ satisfying 
\begin{equation}
\label{cor equation}
\max_{1\leq i\leq d}\Big|\alpha_{i}-\frac{p_{i}}{q}\Big|\leq \frac{1}{q^{1+{1/d}}}.
\end{equation} We say that $\alpha\in\mathbb{R}^{d}$ is a \textit{badly approximable number} if there exists $\kappa(\alpha)>0$ such that  
$$\max_{1\leq i\leq d} \Big|\alpha_{i}-\frac{p_{i}}{q}\Big|> \frac{\kappa(\alpha)}{q^{1+{1/d}}},$$for all $(p_{1},\ldots, p_{d},q)\in\mathbb{Z}^{d}\times \mathbb{N}.$ We denote the set of badly approximable numbers in $\mathbb{R}^{d}$ by $\textbf{Bad}_{d}.$ It is well known that $\mathcal{L}^{d}(\textbf{Bad}_{d})=0$ and $\dim_{H}(\textbf{Bad}_{d})=d.$ Here $\mathcal{L}^{d}(\cdot)$ denotes the $d$-dimensional Lebesgue measure. The fact that $\mathcal{L}^{d}(\textbf{Bad}_{d})=0$ is a consequence of Khintchine's theorem, see \cite[Page 60]{Schmidt book}. See \cite{Schmidt} for the Hausdorff dimension result. In this paper we generalise the result $\mathcal{L}^{d}(\textbf{Bad}_{d})=0$ to more general sequences of balls. We now give details of this generalisation.

Let $B(x,r)$ denote the closed $d$-dimensional Euclidean ball centred at $x\in\mathbb{R}^{d}$ with radius $r>0$. Given a countable set of balls$\{B(x_{i},r_{i})\}_{i=1}^{\infty}$, the lim-sup set associated to $\{B(x_{i},r_{i})\}_{i=1}^{\infty}$ is defined as: $$\limsup B(x_{i},r_{i})=\bigcap_{n=1}^{\infty}\bigcup_{i=n}^{\infty} B(x_{i},r_{i})=\Big\{\alpha\in\mathbb{R}^{d}: \alpha \textrm{ is contained in infinitely many } B(x_{i},r_{i})\Big\}.$$ In what follows we will always assume that our set of balls is such that Lebesgue almost every $\alpha\in\limsup B(x_{i},r_{i}).$  Without this assumption our analysis is degenerate. We say that $\alpha\in\mathbb{R}^{d}$ is a  \textit{badly approximable number with respect to $\{B(x_{i},r_{i})\}_{i=1}^{\infty}$} if there exists $\kappa(\alpha)>0$ such that $\alpha\notin \limsup B(x_{i},\kappa(\alpha)r_{i}).$ In other words, $\alpha$ is a badly approximable number with respect to $\{B(x_{i},r_{i})\}_{i=1}^{\infty}$ if there exists $\kappa(\alpha)>0$ and $N(\alpha)\in\mathbb{N}$ such that $\alpha\notin B(x_{i},\kappa(\alpha)r_{i})$ for all $i\geq N(\alpha).$ We denote the set of badly approximable numbers with respect to $\{B(x_{i},r_{i})\}_{i=1}^{\infty}$ by $\textbf{Bad}(\{B(x_{i},r_{i})\}_{i=1}^{\infty}),$ i.e., 
\begin{align*}
\textbf{Bad}(\{B(x_{i},r_{i})\}_{i=1}^{\infty}):=\Big\{\alpha\in \mathbb{R}^{d}:& \textrm{ there exists } \kappa(\alpha)>0, N(\alpha)\in\mathbb{N} \textrm{ such that } \alpha\notin B(x_{i},\kappa(\alpha)r_{i})\\
&\textrm{ for all } i\geq N(\alpha)\Big\}.
\end{align*} 

Before stating our main result it is necessary to introduce a technical condition. We say that $\{B(x_{i},r_{i})\}_{i=1}^{\infty}$ is \textit{shrinking locally}, if for any ball $B(y,r)$ and $\epsilon>0,$ if we let $\{\hat{B}(x_{j},r_{j})\}_{j=1}^{\infty}$ denote those elements of $\{B(x_{i},r_{i})\}_{i=1}^{\infty}$ that intersect $B(y,r),$ then there exists finitely many $\hat{B}(x_{j},r_{j})$ satisfying $r_{j}>\epsilon.$ We remark that $\{\hat{B}(x_{j},r_{j})\}_{j=1}^{\infty}$ is always infinite under the assumption almost every $\alpha\in\limsup B(x_{i},r_{i}).$ Our main result is the following.

\begin{thm}
\label{main thm}
Assume $\{B(x_{i},r_{i})\}_{i=1}^{\infty}$ is shrinking locally and almost every $\alpha\in\limsup B(x_{i},r_{i})$. Then $\mathcal{L}^{d}(\textbf{Bad}(\{B(x_{i},r_{i})\}_{i=1}^{\infty}))=0.$
\end{thm}

The following corollary is an interesting consequence of Theorem \ref{main thm}.
\begin{cor}
Assume $\{B(x_{i},r_{i})\}_{i=1}^{\infty}$ is shrinking locally. Then almost every $\alpha\in\limsup B(x_{i},r_{i})$ if and only if for all $\kappa>0$ almost every $\alpha\in\limsup B(x_{i},\kappa r_{i})$
\end{cor}

The following example demonstrates that the shrinking locally property is in fact essential. 
\begin{example}
Let $(x_{i})_{i=1}^{\infty}$ be an infinite sequence whose entries are all elements of $\mathbb{Z}^{2},$ moreover assume that every element of $\mathbb{Z}^{2}$ occurs infinitely often in $(x_{i})_{i=1}^{\infty}$. Consider the set of balls $\{B(x_{i},2)\}_{i=1}^{\infty},$ this set is not shrinking locally and $\limsup B(x_{i},2)=\mathbb{R}^{2}.$  It is obvious that $\limsup B(x_{i},2\kappa)$ does contain almost every $\alpha\in \mathbb{R}^{2}$ for $\kappa$ sufficiently small, in fact $\textbf{Bad}(\{B(x_{i},2)\}_{i=1}^{\infty})=\mathbb{R}^{2}\setminus \mathbb{Z}^{2}$. In which case the conclusion of Theorem \ref{main thm} definitely does not hold.
\end{example}
To prove Theorem \ref{main thm} it would suffice to prove an analogous result with  $\textbf{Bad}(\{B(x_{i},r_{i})\}_{i=1}^{\infty})$ replaced by the set of $\alpha \in \mathbb{R}^{d}$ for which there exists $\kappa(\alpha)>0$ such that  $\alpha\notin B(x_{i},\kappa(\alpha)r_{i})$ for all $i\in\mathbb{N}.$ This follows from a straightforward argument. However, we choose not to make use of this fact and stick with our original definition of $\textbf{Bad}(\{B(x_{i},r_{i})\}_{i=1}^{\infty}).$ The proof is marginally more complicated but the author believes that this approach is more instructive.

In the statement of Theorem \ref{main thm} we could replace the set of balls with a set of cubes or other $d$-dimensional objects and still have the same conclusion. The reason we chose to phrase Theorem \ref{main thm} in terms of balls is that it will simplify some of our later exposition.

The proof of Theorem \ref{main thm} will be given in Section $2$. In Section $3$ we explain how Theorem \ref{main thm} recovers the aforementioned result from Diophantine approximation that the set of badly approximable numbers has Lebesgue measure zero.

\section{Proof of Theorem \ref{main thm}}
The proof of Theorem \ref{main thm} will make use of techniques from geometric measure theory. In particular, we will use the Lebesgue density theorem and the Vitali covering lemma. The Lebesgue density theorem states that if $E\subset \mathbb{R}^{d}$ is a measurable set, then for almost every $x\in E$ $$\lim_{r\to 0}\frac{\mathcal{L}^{d}(E\cap B(x,r))}{\mathcal{L}^{d}(B(x,r))}=1.$$ It is a consequence of the Lebesgue density theorem that if $\limsup_{r\to 0}\frac{\mathcal{L}^{d}(E\cap B(x,r))}{\mathcal{L}^{d}(B(x,r))}<1,$ for all $x\in E,$ then $\mathcal{L}^{d}(E)=0.$ This will be the strategy we employ when we show the Lebesgue measure of $\textbf{Bad}(\{B(x_{i},r_{i})\}_{i=1}^{\infty})$ is zero. The Vitali covering lemma states that if $\{B(x_{i},r_{i})\}_{i=1}^{M}$ is a finite set of balls, then it has a subset of disjoint balls $\{\hat{B}(x_{j},r_{j})\}_{j=1}^{N}$ such that $\cup_{i=1}^{M} B(x_{i},r_{i})\subseteq \cup_{j=1}^{N}\hat{B}(x_{j},3r_{j}).$ Proofs of the Lebesgue density theorem and the Vitali covering lemma can be found in \cite{Falconer}. We now prove one technical lemma before giving our proof of Theorem \ref{main thm}.

\begin{lemma}
\label{lemma}
 Let $\{B(x_{i},r_{i})\}_{i=1}^{\infty}$ be a finite or infinite countable set of $d$-dimensional Euclidean balls and $0<\delta\leq 1.$ Then $$\mathcal{L}^{d}\Big(\bigcup_{i=1}^{\infty}B(x_{i},\delta r_{i})\Big)\geq K(\delta,d)\mathcal{L}^{d}\Big(\bigcup_{i=1}^{\infty}B(x_{i},r_{i})\Big).$$ Where $K(\delta,d)$ is some strictly positive constant depending only on $\delta$ and $d.$
\end{lemma}

The proof of Lemma \ref{lemma} will be split into two parts. First of all we prove the statement for $d=1$ and then for $d\geq 1$. The proof of Lemma \ref{lemma} for $d\geq 1$ is when we use the Vitali covering lemma. The reason for two proofs of the case when $d=1$ is our initial method yields an optimal value for $K(\delta,1)$ and does not require the Vitali covering lemma. 
\begin{proof}[Proof of Lemma \ref{lemma} when $d=1$]
As the Lebesgue measure is continuous from below it is sufficent to prove this lemma only in the case of a finite set of balls. Let $\{B(x_{i},r_{i})\}_{i=1}^{M}$ be such a finite set. Consider $\cup_{i=1}^{M} B(x_{i},\delta r_{i}),$ this set is made up of finitely many disjoint connected components, we denote the set of these components by $\{C_{l}\}_{l=1}^{n}.$ Clearly $C_{l}=\cup_{j=1}^{N_{l}} B(x_{j}^{l},\delta r_{j}^{l})$ for some set of balls $\{B(x_{j}^{l},\delta r_{j}^{l})\}_{j=1}^{N_{l}}.$ If $\cup_{j=1}^{N_{l}} B(x_{j}^{l},\delta r_{j}^{l})$ is a component then $\cup_{j=1}^{N_{l}} B(x_{j}^{l},r_{j}^{l})$ is an interval. This is clear as we are considering larger balls. Let $B(x_{j'}^{l},r_{j'}^{l})$ and $B(x_{j''}^{l},r_{j''}^{l})$ be such that $\cup_{j=1}^{N_{l}} B(x_{j}^{l},r_{j}^{l})=[x_{j'}^{l}-r_{j'}^{l},x_{j''}^{l}+r_{j''}^{l}].$ It is a straightforward exercise to show that $x_{j'}^{l}\leq x_{j''}^{l}.$ Both $x_{j'}^{l}-\delta r_{j'}^{l}, x_{j''}^{l}+\delta r_{j''}^{l}\in C_{l},$ and as $C_{l}$ is a connected component in $\mathbb{R}$ and therefore an interval we have 
\begin{equation}
\label{equation 1}
\mathcal{L}^{d}(C_{l})\geq (x_{j''}^{l}-x_{j'}^{l})+\delta(r_{j''}^{l}-r_{j'}^{l})\geq \delta(x_{j''}^{l}-x_{j'}^{l}+r_{j''}^{l}-r_{j'}^{l})=\delta \mathcal{L}^{d}\Big(\bigcup_{j=1}^{N_{l}} B(x_{j}^{l},r_{j}^{l})\Big).
\end{equation} Making use of the estimate provided by (\ref{equation 1}) we observe the following: 
\begin{align*}
\mathcal{L}^{d}\Big(\bigcup_{i=1}^{M}B(x_{i}, r_{i})\Big)=\mathcal{L}^{d}\Big(\bigcup_{l=1}^{n}\bigcup_{j=1}^{N_{l}}B(x_{j}^{l},r_{j}^{l})\Big)&\leq  \sum_{l=1}^{n} \mathcal{L}^{d}\Big( \bigcup_{j=1}^{N_{l}}B(x_{j}^{l},r_{j}^{l})\Big)\\
&\leq \frac{1}{\delta}\sum_{l=1}^{n}  \mathcal{L}^{d}(C_{l})\\
& = \frac{1}{\delta} \mathcal{L}^{d}\Big(\bigcup_{i=1}^{M} B(x_{i},\delta r_{i})\Big).
\end{align*}In the final equality we have used the fact that the $C_{l}$'s are disjoint and their union equals $\cup_{i=1}^{M} B(x_{i},\delta r_{i}).$ It follows that in the case where $d=1$ we can take $K(\delta,1)=\delta.$
\end{proof}To see that $\delta$ is the optimal value for $K(\delta,1)$ consider the case of a single ball. 
\begin{proof}[Proof of Lemma \ref{lemma} when $d\geq 1$]
As in the case where $d=1$ it is sufficent to prove our result for a finite sets of balls. Let $\{B(x_{i},r_{i})\}_{i=1}^{M}$ be a finite set of balls and $\{\hat{B}(x_{j},r_{j})\}_{j=1}^{N}$ be the disjoint subset of balls guaranteed by the Vitali covering lemma. This subset satisfies $\cup_{i=1}^{M} B(x_{i},r_{i})\subseteq \cup_{j=1}^{N}\hat{B}(x_{j},3r_{j})$. Therefore
\begin{align*}
\mathcal{L}^{d}\Big(\bigcup_{i=1}^{M} B(x_{i},r_{i})\Big)\leq \mathcal{L}^{d}\Big(\bigcup_{j=1}^{N}\hat{B}(x_{j},3r_{j})\Big)&\leq \sum_{j=1}^{N}\mathcal{L}^{d}(\hat{B}(x_{j},3r_{j}))\\
&\leq \Big(\frac{3}{\delta}\Big)^{d}\sum_{j=1}^{N}\mathcal{L}^{d}(\hat{B}(x_{j},\delta r_{j}))\\
& = \Big(\frac{3}{\delta}\Big)^{d} \mathcal{L}^{d}\Big(\bigcup_{j=1}^{N}\hat{B}(x_{j},\delta r_{j})\Big)\\
&\leq \Big(\frac{3}{\delta}\Big)^{d}\mathcal{L}^{d}\Big(\bigcup_{i=1}^{M} B(x_{i},\delta r_{i})\Big).
\end{align*}In the above we have used the fact that $\bigcup_{j=1}^{N}\hat{B}(x_{j},\delta r_{j})$ is a disjoint union. Taking $K(\delta,d)=(\frac{\delta}{3})^{d}$ yields our result.
\end{proof}
We anticipate that $\delta^{d}$ will in fact be the optimal value for $K(\delta,d)$ for all $d\geq 1$. Now we have proved Lemma \ref{lemma} we are able to prove Theorem \ref{main thm}.
\begin{proof}[Proof of Theorem \ref{main thm}]
To begin with we introduce the following collection of sets, given $M,N\in\mathbb{N}$ let $$B(N,M)=\Big\{\alpha\in \mathbb{R}^{d}: \textrm{ for all } i\geq N \textrm{ we have }\alpha\notin B\Big(x_{i},\frac{r_{i}}{M}\Big)\Big\}.$$
It is obvious that $$\textbf{Bad}(\{B(x_{i},r_{i})\}_{i=1}^{\infty})=\bigcup_{N=1}^{\infty}\bigcup_{M=1}^{\infty}B(N,M).$$ Therefore to show $\mathcal{L}^{d}(\textbf{Bad}(\{B(x_{i},r_{i})\}_{i=1}^{\infty}))=0$ it is sufficient to prove $\mathcal{L}^{d}(B(N,M))=0$ for any $N,M\geq 1.$ We show this to be the case via an application of the Lebesgue density theorem. Let us now fix $N,M\geq 1,$ $y\in B(N,M)$ and $r>0.$ We now state three properties of the set $\{B(x_{i},r_{i})\}_{i=1}^{\infty}.$ These properties will allow us to obtain a useful subset of $\{B(x_{i},r_{i})\}_{i=1}^{\infty}$. Each element of this subset when scaled by a factor $M^{-1}$ will be contained in $B(N,M)^{c}\cap B(y,r),$ and the Lebesgue measure of the union of the balls in this subset will be comparable to the that of $B(y,r).$

\begin{itemize}
\item For each $L\geq 1$ we have $\cup_{i=L}^{\infty}B(x_{i},r_{i})$ equals $\mathbb{R}^{d}$ up to a set of measure zero. This is obvious as $\limsup B(x_{i},r_{i})$ equals $\mathbb{R}^{d}$ up to a set of measure zero.
\vspace{3.6mm}
\item For each $L\geq N$ we have $\bigcup_{i=L}^{\infty}B(x_{i},\frac{r_{i}}{M})\subset B(N.M)^{c}.$
\vspace{3.6mm}
\item For $L\geq 1$ let $\{\hat{B}^{L}(x_{j},r_{j})\}_{j=1}^{\infty}$ denote those elements of $\{B(x_{i},r_{i})\}_{i=L}^{\infty}$ which intersect $B(y,r).$ Given an $\epsilon>0,$ then by the shrinking locally property we can pick $L$ sufficiently large such that $\hat{B}^{L}(x_{j},r_{j})$ has radius less than $\epsilon$ for every $j\in \mathbb{N}.$
\end{itemize}

Let $\delta>0$ be some arbitrary positive constant. Using the above properties of $\{B(x_{i},r_{i})\}_{i=1}^{\infty},$ we may assert that by considering $L$ sufficiently large we can choose a subset of $\{\hat{B}^{L}(x_{j},r_{j})\}_{j=1}^{\infty},$ which we shall denote by $\{\tilde{B}(x_{k},r_{k})\}_{k=1}^{\infty},$ with the following properties:
\begin{itemize}
 \item $(1-\delta)\mathcal{L}^{d}(B(y,r))\leq \mathcal{L}^{d}(\bigcup_{k=1}^{\infty}\tilde{B}(x_{k},r_{k}))\leq \mathcal{L}^{d}(B(y,r)).$
\vspace{3.6mm}
\item $\tilde{B}(x_{k},\frac{r_{k}}{M})\subset B(N,M)^{c}$ for all $k\geq 1.$
\vspace{3.6mm}
\item $\tilde{B}(x_{k},r_{k})\subset B(y,r)$ for all $k\geq 1.$
\end{itemize}

We now use the above properties of $\{\tilde{B}(x_{k},r_{k})\}_{k=1}^{\infty}$ and Lemma \ref{lemma} to obtain the following estimate on the density of $B(N,M)$ at $y:$ 

\begin{align*}
\frac{\mathcal{L}^{d}(B(y,r)\cap B(N,M))}{\mathcal{L}^{d}(B(y,r))}&=1-\frac{\mathcal{L}^{d}(B(y,r)\cap B(N,M)^{c})}{\mathcal{L}^{d}(B(y,r))}\\
&\leq 1-\frac{\mathcal{L}^{d}(\cup_{k=1}^{\infty}\tilde{B}(x_{k},\frac{r_{k}}{M}))}{\mathcal{L}^{d}(B(y,r))}\\
&\leq 1-\frac{K(\frac{1}{M},d)\mathcal{L}^{d}(\cup_{k=1}^{\infty}\tilde{B}(x_{k},r_{k}))}{\mathcal{L}^{d}(B(y,r))}\\
&\leq 1-\frac{K(\frac{1}{M},d)(1-\delta)\mathcal{L}^{d}(B(y,r))}{\mathcal{L}^{d}(B(y,r))}\\
&= 1-K\Big(\frac{1}{M},d\Big)(1-\delta).
\end{align*}As $\delta$ was arbitrary we have that $\limsup_{r\to 0}\frac{\mathcal{L}^{d}(B(y,r)\cap B(N,M))}{\mathcal{L}^{d}(B(y,r))}$ is always at most $1-K(\frac{1}{M},d).$ Which by the Lebesgue density theorem implies $\mathcal{L}^{d}(B(N,M))=0,$ and by our earlier remarks we may conclude $\mathcal{L}^{d}(\textbf{Bad}(\{B(x_{i},r_{i})\}_{i=1}^{\infty}))=0.$
\end{proof}

\section{Recovering $\mathcal{L}^{d}(\textbf{Bad}_{d})=0$}
Showing Theorem \ref{main thm} implies $\mathcal{L}^{d}(\textbf{Bad}_{d})=0$ is fairly straightforward, we include the argument for completion.

Let $C(x,r)$ denote the $d$-dimensional Euclidean cube centred at $x$ with side length $2r.$ Then 
$$\textbf{Bad}_{d}=\bigcup_{n=1}^{\infty}\Big\{\alpha\in\mathbb{R}^{d}: \alpha\notin C\Big(\Big(\frac{p_{1}}{q},\ldots, \frac{p_{d}}{q}\Big),\frac{1}{nq^{1+1/d}}\Big) \textrm{ for all } (p_{1},\ldots,p_{d},q)\in \mathbb{Z}^{d}\times \mathbb{N}\Big\}.$$
The following inclusions are well known: $B(x,r)\subset C(x,r)\subset B(x,\sqrt{d}r).$ These inclusions allow us to interpret $\textbf{Bad}_{d}$ in terms of balls instead of cubes: 
\begin{equation}
\label{inclusion}
\textbf{Bad}_{d}=\bigcup _{n=1}^{\infty}\Big\{\alpha\in \mathbb{R}^{d}: \alpha\notin B\Big(\Big(\frac{p_{1}}{q},\ldots ,\frac{p_{d}}{q}\Big),\frac{\sqrt{d}}{nq^{1+1/d}}\Big) \textrm{ for all }(p_{1},\ldots,p_{d},q)\in \mathbb{Z}^{d}\times \mathbb{N}\Big\}.
\end{equation}
Consider the set of balls $\{B((\frac{p_{1}}{q},\ldots, \frac{p_{d}}{q}),\frac{\sqrt{d}}{q^{1+1/d}})\}_{(p_{1},\ldots,p_{d},q)\in \mathbb{Z}^{d}\times \mathbb{N}}.$ This set of balls clearly satisfies the shrinking locally property. It is a consequence of the corollary of Dirichlet's theorem stated at the beginning, see (\ref{cor equation}), and $C((\frac{p_{1}}{q},\ldots, \frac{p_{d}}{q}),\frac{1}{q^{1+1/d}})\subset B((\frac{p_{1}}{q},\ldots, \frac{p_{d}}{q}),\frac{\sqrt{d}}{q^{1+1/d}}),$ that almost every $\alpha\in \limsup B((\frac{p_{1}}{q},\ldots, \frac{p_{d}}{q}),\frac{\sqrt{d}}{q^{1+1/d}})$. Therefore Theorem \ref{main thm} applies and $\mathcal{L}^{d}(\textbf{Bad}(\{B((\frac{p_{1}}{q},\ldots, \frac{p_{d}}{q}),\frac{\sqrt{d}}{q^{1+1/d}})\}_{(p_{1},\ldots,p_{d},q)\in \mathbb{Z}^{d}\times \mathbb{N}}))=0.$ 

Let us recall that $\textbf{Bad}(\{B((\frac{p_{1}}{q},\ldots, \frac{p_{d}}{q}),\frac{\sqrt{d}}{q^{1+1/d}}\})$ is defined to be the set of $\alpha\in\mathbb{R}^{d}$ for which there exists $\kappa(\alpha)>0$ such that $\alpha$ is in finitely many $B((\frac{p_{1}}{q},\ldots, \frac{p_{d}}{q}),\frac{\kappa(\alpha)\sqrt{d}}{q^{1+1/d}}).$ Therefore, by (\ref{inclusion}) we have $\textbf{Bad}_{d}\subset \textbf{Bad}(\{B((\frac{p_{1}}{q},\ldots, \frac{p_{d}}{q}),\frac{\sqrt{d}}{q^{1+1/d}})\})$ and $\mathcal{L}^{d}(\textbf{Bad}_{d})=0.$ It is not a difficult exercise to prove that in fact $\textbf{Bad}_{d}= \textbf{Bad}(\{B((\frac{p_{1}}{q},\ldots, \frac{p_{d}}{q}),\frac{\sqrt{d}}{q^{1+1/d}})\}).$ This is a consequence of the fact that if $\alpha\in \textbf{Bad}(\{B((\frac{p_{1}}{q},\ldots, \frac{p_{d}}{q}),\frac{\sqrt{d}}{q^{1+1/d}})\})$ then it cannot be in $\mathbb{Q}^{d},$ in which case by considering smaller $\kappa(\alpha)$ if necessary we have that $\alpha\notin B((\frac{p_{1}}{q},\ldots, \frac{p_{d}}{q}),\frac{\kappa(\alpha)\sqrt{d}}{q^{1+1/d}})$ for all $(p_{1},\ldots,p_{d},q)\in \mathbb{Z}^{d}\times \mathbb{N}.$ For a general $\{B(x_{i},r_{i})\}_{i=1}^{\infty}$ we do not necessarily have this equality.
\\

\noindent \textbf{Acknowledgements} This work was initiated at the school on \textit{Dynamics and Number theory} at Durham university. The author would like to thank Lior Fishman and Sanju Velani for useful discussions. This research was supported by the Dutch Organisation for Scientiﬁc Research (NWO) grant number 613.001.022.


\begin{thebibliography}{100}
\bibitem{Falconer} K. Falconer, \textit{The geometry of fractal sets,} Cambridge Tracts in Mathematics, 85. Cambridge University Press, Cambridge, 1986. ISBN: 0-521-25694-1; 0-521-33705-4 
\bibitem{Schmidt} W. Schmidt, \textit{Badly approximable systems of linear forms,} J. Number Theory 1 1969 139--154.
\bibitem{Schmidt book} W. Schmidt, \textit{Diophantine Approximation,} Lecture Notes in Mathematics, 785. Springer, Berlin, 1980. x+299 pp. ISBN: 3-540-09762-7 
\end{thebibliography}
\end{document}